\documentclass[a4paper]{article}

\usepackage{geometry}
\usepackage[utf8]{inputenc}
\usepackage{subfigure}
\usepackage{amsfonts,amssymb,amsmath,dsfont,amsthm}
\usepackage{enumitem}
\usepackage{xspace}
\usepackage{graphicx}
\usepackage[colorlinks=true,citecolor=blue]{hyperref}
\usepackage{bm}

\newtheorem{theorem}{Theorem}
\newtheorem{lemma}{Lemma}
\newtheorem{corollary}{Corollary}

\newcommand{\proba}{\mathds{P}}
\newcommand{\esp}{\mathds{E}}
\newcommand{\bigO}{\mathcal{O}}

\newcommand{\smallo}{o}

\newcommand{\mA}{\mathcal{A}}
\newcommand{\mB}{\mathcal{B}}
\newcommand{\mC}{\mathcal{C}}

\newcommand{\mF}{\mathcal{F}}
\newcommand{\mH}{\mathcal{H}}

\newcommand*{\ie}{\textit{i.e.}\@\xspace}

\newcommand{\sg}{\operatorname{SG}}
\newcommand{\mg}{\operatorname{MG}}
\newcommand{\Patch}{\operatorname{Patch}}

\title{Threshold functions for small subgraphs: an analytic approach}

\author{
Gwendal Collet\footnote{Institute of Discrete Mathematics and Geometry, TU Wien (Austria). 
Partially supported by the Austrian Science Foundation FWF, grant SFB F50-02.
},
\'Elie de Panafieu\footnote{Nokia Bell Labs and LINCS (France). 
This work was partially founded by the Austrian Science Fund (FWF) grant F5004, the Amadeus program and the PEPS HYDrATA.
},
Dani\`ele Gardy\footnote{DAVID Laboratory, University of Versailles Saint Quentin (France). 
Partially supported by the Amadeus project 33697ZK \emph{Threshold problems and phase transitions in graph-like structures} (2015--16), by the PICS project \emph{Constraint analysis through analytic combinatorics} (2017--19), and by the ANR-MOST MetaConc (2015--19).
},\\
Bernhard Gittenberger\footnote{Institute of Discrete Mathematics and Geometry, TU Wien, Wiedner Hauptstr. 8--10/104, 1040 Wien, Austria. Partially supported by the Austrian Science Foundation FWF, grant SFB F50-03 and the \"OAD grant Amadée F01/2015.}
,
Vlady Ravelomanana\footnote{IRIF, University of Paris 7 (France). 
Partially supported by the Amadeus project 33697ZK \emph{Threshold problems and phase transitions in graph-like structures} (2015--16), by the project \emph{Combinatorics in Paris} (2014--17)  and by the PICS project \emph{Constraint analysis through analytic combinatorics} (2017--19).
}}

\begin{document}
\maketitle

\begin{abstract}
We revisit the problem of counting the number of copies of a fixed graph in a random graph or multigraph, including the case of constrained degrees.
Our approach relies heavily on analytic combinatorics and on the notion of \emph{patchwork} to describe the possible overlapping of copies.

This paper is a version, extended to include proofs, of the paper with the same title to be presented at the Eurocomb 2017 meeting.

\end{abstract}

\section{Introduction}

Since the introduction of the random graph models $G(n,\, m)$ and $G(n,\, p)$ 
by Erd\H os and R\'enyi~\cite{ER60} in 1960, one of the most studied parameters 
is the number $X_F$ of subgraphs isomorphic to a given graph $F$.
By the asymptotic equivalence between $G(n,\, p)$ and $G(n,\, m)$,
results from one model can be rigorously translated into the other one.
Erd\H{o}s and R\'enyi derived the threshold for $\{X_F > 0\}$
when $F$ is a \emph{strictly balanced} graph (see definition next page),
and Bollob\'as~\cite{Bo81} generalized their result to any graph $F$.
%
Ruci{\'n}ski~\cite{Ru88} proved that $X_F$
is asymptotically normal beyond the threshold,
and follows a Poisson law at the threshold iff $F$ is strictly balanced.
Then Janson, Oleszkiewicz and Ruci\'nski~\cite{JaOlRu04}
developed a moment-based method for estimating
$\proba( X_F \geq (1+\varepsilon) \esp(X_F) )$.
The notion of \emph{strongly balanced graphs},
introduced by Ruci{\'n}ski and Vince in~\cite{RuVi86},
plays a  key role in obtaining the results mentioned above.

Recently, there has been an increasing interest
in the study of constrained random graphs,
such as given degree sequences or regular graphs;
the number of given subgraphs in such structures has been also studied.
E.g., Wormald~\cite{wormald-survey} proved that the number of short cycles
in these structures asymptotically follows a Poisson distribution;
using a multi-dimensional saddle-point approach, McKay~\cite{mckay11}
studied the structure of a random graph with given degree sequence,
including the probability of a given subgraph or induced subgraph.

Our goal is to revisit (part of) these results
through analytic combinatorics and extensive use of generating functions (g.f.).
Ours is not the first paper that approaches graph problems with these tools.
Early such work was by McKay and Wormald
(see, e.g., \cite{McKayWormald90} for the enumeration
of graphs with a specified degree sequence);
an important development was the study of planar graphs
by Gim\'enez and Noy~\cite{GimenezNoy09}, followed by several papers in the same direction; 
see also a recent paper by Drmota, Ramos and Ru\'e~\cite{DrmotaRamosRue}
about the limiting distribution of the number of copies of a subgraph in subcritical graphs.

In the rest of this section, we give formal definitions of our model and the objects we are interested in.
Then  we address the problem of evaluating the number of subgraphs
in Section~\ref{sec:subgraphs};
finally some of those results are extended
to graphs and multigraphs with degree constraints
in Section~\ref{sec:degree_constraints}.

\medskip
In the rest of this section, we give formal definitions of our model and the objects we are interested in.
Then  we address the problem of evaluating the number of subgraphs, applying analytic combinatorics tools, \ie generating function manipulations, in Section~\ref{sec:subgraphs}.
Finally some of those results are extended to graphs and multigraphs with degree constraints in Section~\ref{sec:degree_constraints}.

\paragraph{Model and definitions.}

Most of the following definitions come from \cite{ER60} and \cite{Bo81}.
A graph $G$ is a pair $(V(G),E(G))$,
where $V(G)$ denotes the set of labeled vertices,
and $E(G)$ the set of edges.
Each edge is a unoriented pair of distinct vertices,
thus loops 
and multiple edges 
are forbidden.
An $(n,m)$-graph is a graph with $n$ vertices,
labeled from $1$ to $n$, and $m$ edges.
A graph $F$ is a \emph{subgraph} of $G$
if $V(F) \subset V(G)$ and $E(F) \subset E(G)$.
We then write $F \subset G$.
Two graphs $F$, $G$ are \emph{isomorphic}
if there exists a bijection from $V(F)$ to $V(G)$
that induces a bijection between $E(F)$ and $E(G)$.
An \emph{$F$-graph} is a graph isomorphic to $F$,
an \emph{$F$-subgraph} of $G$ is a subgraph of $G$ that is an $F$-graph,
and $G[F]$ denotes the number of subgraphs of $G$ that are $F$-graphs.
Given a graph family $\mF$,
an \emph{$\mF$-graph} is an $F$-graph for some $F \in \mF$.
The \emph{density} $d(G)$ of a graph $G$ is the ratio
between its number of edges and of vertices.
A graph is \emph{strictly balanced}
if its density is larger than the density of its strict subgraphs.
The \emph{essential density} $d^\star(G)$ of $G$
is the highest density of its subgraphs
\[
    d^\star(G) = \max_{H \subset G} d(H).
\]
To any graph family $\mF$, we associate the generating function
\[
    F(z,w) = \sum_{n,m \geq 0} \mF_{n,m} w^m \frac{z^n}{n!},
\]
where $\mF_{n,m}$ denotes the number of $(n,m)$-graphs isomorphic to a graph from $\mF$.

\newpage

        \section{Number of subgraphs in a random graph} \label{sec:subgraphs}

   \paragraph{Graphs with one distinguished subgraph.}

\begin{theorem} \label{th:distinguished}
The number of $(n,m)$-graphs where one $\mF$-subgraph is distinguished is
\begin{equation} \label{eq:distinguished}
    n! [z^n w^m] F \left( z, \frac{w}{1+w} \right) e^z (1+w)^{\binom{n}{2}}
    \sim \binom{\binom{n}{2}}{m} F \left( n, \frac{m}{\binom{n}{2}} \right),
\end{equation}
where the asymptotics holds when
$F(z,w)$ is an entire function,
$m$ tends to infinity with $n$ while $m = \smallo(n^2)$,
and $F(nz,mw/\binom{n}{2}) / F(n,m/\binom{n}{2})$ converges uniformly on any compact set
to an analytic function.
\end{theorem}

\begin{proof}
A graph on $n$ vertices where one $\mF$-subgraph is distinguished is
a copy of a graph $F$ from $\mF$,
a set of additional vertices,
and a set of additional edges.
Those edges can link any pair from the $n$ vertices,
except the pairs already linked in $F$.
The Symbolic Method (see \cite{FS09}) translates
this combinatorial description into
the generating function expression of the theorem.
The asymptotics is then extracted using a saddle-point method.
\end{proof}

Let $H$ denote a densest subgraph of $F$.
Theorem~\ref{th:distinguished} is now applied
with $\mF$ equal to the family of the $H$-graphs.
Dividing both sides of Equation~\eqref{eq:distinguished}
by the total number $\binom{\binom{n}{2}}{m}$ of $(n,m)$-graphs,
we obtain a new proof for the following classic result of \cite{ER60, Bo81}.

\begin{corollary} \label{th:distinguished_esp}
Denote by $\ell^{\star}$ and $d^{\star}$ the number of edges and density
of a densest subgraph of $F$,
and consider a random $(n,m)$-graph $G$
with $m = \bigO(n^{\alpha})$ for some fixed $0 < \alpha < 2$,
then
\[
    \esp(G[F]) = \bigO(n^{\ell^{\star}(\alpha - 2 + 1/d^\star)}).
\]
Thus, for any $\alpha < 2 - 1/d^\star$, $G[F] = 0$ almost surely.
\end{corollary}

   \paragraph{Graphs with marked subgraphs.}\hfill
   
Given a graph $F$, an \emph{$F$-patchwork} $P$ is
a set of distinct $F$-graphs $\{F_1, \ldots, F_t\}$
that might share vertices and edges,
and such that the pair $\left( \cup_{i=1}^t V(F_i), \cup_{i=1}^t E(F_i) \right)$
is a graph, denoted by $G(P)$.
This notion is illustrated in Figure~\ref{fig:patchwork}.
\begin{figure}
\begin{center}
\includegraphics[width=11cm]{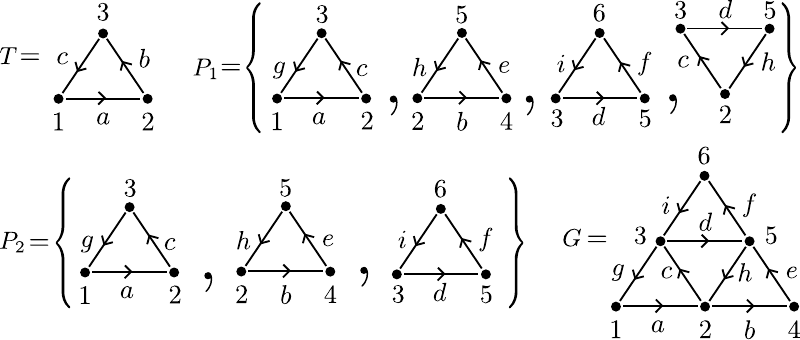}
\caption{A graph $T$ and two $T$-patchworks $P_1$ and $P_2$ that reduce to the same graph $G(P_1) = G(P_2) = G$.} \label{fig:patchwork}
\end{center}
\end{figure}
Let $\Patch_{F,n,m,t}$ denote the number
of $F$-patchworks that are composed of $t$ $F$-graphs,
and such that $G(P)$ is an $(n,m)$-graph.
Then the generating function of $F$-patchworks is defined as
\[
    \Patch_F(z,w,u) =
    \sum_{n,m,t \geq 0}
    \Patch_{F,n,m,t}
    u^t
    w^m
    \frac{z^n}{n!}.
\]

\begin{theorem} \label{th:subgraphs}
The number $\sg_{n,m,t}^F$ of $(n,m)$-graphs that contain exactly $t$ $F$-subgraphs is
\begin{equation} \label{eq:subgraphs}
    \sg_{n,m,t}^F =
    n! [z^n w^m u^t] \Patch_{F} \left( z, \frac{w}{1+w}, u-1 \right) e^z (1+w)^{\binom{n}{2}}.
\end{equation}
\end{theorem}

\begin{proof}
We introduce the generating function of $(n,m)$-graphs,
where the variable $u$ marks the total number of $F$-subgraphs
\[
    \sg^F_{n,m}(u) =
    \sum_{\text{$(n,m)$-graph $G$}}
    u^{G[F]} =
    \sum_{t \geq 0} \sg_{n,m,t}^F u^t,
\]
and apply the following inclusion-exclusion argument.
$\sg^F_{n,m}(u+1)$ is the generating function of $(n,m)$-graphs
where each $F$-subgraph is either marked, or left unmarked.
By definition, the set of marked $F$-subgraphs is a patchwork,
which is distinguished in the graph.
Thus, we can apply Theorem~\ref{th:distinguished},
where $F(z,w)$ is replaced by $\Patch_F(z,w,u)$
\[
    \sg^F_{n,m}(u+1) =
    n! [z^n w^m]
    \Patch_F\left(z,\frac{w}{1+w},u \right)
    e^z (1+w)^{\binom{n}{2}}.
\]
We then replace $u$ with $u-1$ and extract the coefficient $[u^t]$.
\end{proof}

For a general graph $F$, we do not have an explicit expression
for the generating function of $F$-patchworks.
However, partial information is enough to address some interesting problems.
The following theorem was first derived by \cite{Bo81}.

\begin{theorem} \label{th:subgraphs_limit}
Let $F$ denote a strictly balanced graph of density $d$,
with $\ell$ edges and $\mathfrak{a}$ automorphisms,
and assume $m \sim c n^{2-1/d}$
for some positive constant $c$.
Then the number of $F$-subgraphs in a random $(n,m)$-graph $G$
follows a Poisson limit law of parameter $\lambda = (2c)^{\ell}/\mathfrak{a}$,
\ie for any nonnegative integer $t$,
\[
    \lim_{n \to +\infty} \proba(G[F] = t) = \frac{\lambda^t}{t!} e^{-\lambda}.
\]
\end{theorem}

\begin{proof}
As observed by \cite{ER60}, since $F$ is strictly balanced,
any graph $H$ containing two non-disjoint $F$-graphs
has a higher essential density than $F$.
According to Corollary~\ref{th:distinguished_esp},
the random graph $G$ of the theorem almost surely
contains no such $H$-subgraph.
Following this intuition,
one can prove that only patchworks of disjoint $F$-graphs
have a nonnegligible contribution in Equation~\eqref{eq:subgraphs}.
So we replace $\Patch(z,w,u)$ with $e^{u F(z,w)}$ and obtain
\[
    \sg^F_{n,m,t} \sim
    n! [z^n w^m u^t] e^{(u-1) F(z,w/(1+w))} e^z (1+w)^{\binom{n}{2}}
    \sim
    \binom{\binom{n}{2}}{m}
    [u^t] e^{(u-1) F \left(n, m/\binom{n}{2} \right)},
\]
where the asymptotics comes from Theorem~\ref{th:distinguished}.
Dividing by the total number of $(n,m)$-graphs
and observing $F(n, m/\binom{n}{2}) = \lambda$ finishes the proof.
\end{proof}

        \section{Small subgraphs in graphs with degree constraints} \label{sec:degree_constraints}

We consider now $(n,m,D)$-graphs,
which are $(n,m)$-graphs where all vertices
have their degree in the set $D$.
In the following, $D$ contains at least two integers.
We restrict our study to the case where
$m$ goes to infinity with $n$ in such a way that
$\frac{2m}{n}$ has a limit in $]\min(D), \max(D)[$.
Since the sum of the degrees is twice the number of edges,
if $\frac{2m}{n}$ reaches one of those bounds,
the corresponding $(n,m,D)$-graphs are regular
(a case already treated in the literature),
while if $\frac{2m}{n}$ is outside the interval,
there exist no $(n,m,D)$-graphs.
Finally, to shorten the theorems, we assume $\gcd(d - \min(D)\ |\ d \in D) = 1$.
The generating function of the set $D$ is $\Delta(x) = \sum_{d \in D} \frac{x^d}{d!}$,
and we define $\chi = \chi_{\frac{m}{n}}$ as
the unique positive solution (see Note~IV.46 of \cite{FS09}) of
\[
	\frac{\chi \Delta'(\chi)}{\Delta(\chi)} = \frac{2m}{n}.
\]
As observed by \cite{BC78, Wo78, Bollobas, EdPR16},
\emph{multigraphs} are easier to analyze than graphs
when considering degree constraints.
A multigraph $G$ is a pair $(V(G), E(G))$
where $V(G)$ denotes the set of labeled vertices,
and $E(G)$ the set of labeled oriented edges.
Each edge is thus an oriented pair of vertices,
and loops and multiple edges are allowed.
The definitions on graphs are extended naturally to multigraphs.
Given a multigraph family $\mF$,
let $\mF_{n,m,(d_0, d_1, \ldots)}$ denote
the number of $(n,m)$-multigraphs
with $d_j$ vertices of degree $j$, for all $j \geq 0$,
that are isomorphic to some multigraph from $\mF$.
We associate to the family $\mF$ the generating function
\[
	F(z,w,(\delta_0, \delta_1, \ldots)) =
    \sum_{n,m,d_0, d_1, \ldots}
    \mF_{n,m,(d_0, d_1, \ldots)}
    \bigg( \prod_{s \geq 0} \delta_s^{d_s} \bigg)
    \frac{w^m}{2^m m!}
    \frac{z^n}{n!}.
\]

\begin{theorem} \label{th:distinguished_degree_constraints}
With the previous notations and conventions,
given a multigraph family $\mF$,
the number of $(n,m,D)$-multigraphs
where one $\mF$-subgraph is distinguished is
\begin{equation} \label{eq:distinguished_degree_constraints}
	n! 2^m m! [z^n w^m]
    \sum_{j \geq 0}
    (2j)! [x^{2j}]
    F \left(z,w,(\Delta(x), \Delta'(x), \Delta''(x), \ldots) \right)
    e^{z \Delta(x)}
    \frac{w^j}{2^j j!}.
\end{equation}
If $F \left(\frac{n z}{\Delta(x \chi)}, \frac{w (x \chi)^2}{2m},(\Delta(x \chi), \Delta'(x \chi), \ldots) \right) /
F \left(\frac{n}{\Delta(\chi)}, \frac{\chi^2}{2m},(\Delta(\chi), \Delta'(\chi), \ldots) \right)$
converges uniformly on any compact set to an analytic function,
and $\mg_{n,m,D}$ denotes the total number of $(n,m,D)$-multigraphs,
then the asymptotics of Equation~\eqref{eq:distinguished_degree_constraints} is
\[
    \mg_{n,m,D}
    F \left(\frac{n}{\Delta(\chi)}, \frac{\chi^2}{2m},(\Delta(\chi), \Delta'(\chi), \ldots) \right).
\]
\end{theorem}

\begin{proof}
This result is obtained as a combination
of the proof of Theorem~\ref{th:distinguished}
and the work of \cite{EdPR16}.
In particular, those authors derived an expression for the number of $(n,m,D)$-multigraphs
\[
    \mg_{n,m,D} = (2m)! [x^{2m}] \Delta(x)^n.
\]
For the asymptotics, the sum is rewritten as an integral
\[
    \frac{n! 2^m m!}{\sqrt{2\pi}}
    [z^n w^m]
    \int_{-\infty}^{+\infty}
    F(z,\sqrt{w} t,(\Delta(\sqrt{w} t), \Delta'(\sqrt{w} t), \ldots))
    e^{z \Delta(\sqrt{w} t)}
    e^{-t^2/2}
    dt,
\]
and a saddle-point method is applied.
\end{proof}

A direct consequence of the previous theorem is
the counterpart of Corollary~\ref{th:distinguished_esp}.

\begin{corollary} \label{th:distinguished_esp_degree_constraints}
Denote by $\ell^{\star}$ and $d^{\star}$ the number of edges and density
of a densest subgraph of the multigraph $F$,
and consider a random $(n,m,D)$-multigraph $G$
then
\[
    \esp(G[F]) = \bigO(n^{\ell^{\star}(1/d^\star - 1)}).
\]
\end{corollary}

As stated at the beginning of this section,
we consider random $(n,m,D)$-multigraphs
with a number $m$ of edges that grows linearly
with the number $n$ of vertices.
In that case, $\chi$ as a finite positive limit.
Thus, the condition of the following theorem
is satisfied only when $F$ is a cycle.
However, in a future extension of this work,
we plan to consider the case where $\frac{2m}{n}$ goes to infinity
(when $D$ is infinite).
In this more general setting,
other subgraphs than cycles will appear,
but the condition should remain as stated here.

\begin{theorem} \label{th:multigraphs_limit_law_degree_constraints}
Let $F$ denote a strictly balanced
$(k,\ell)$-multigraph with $\mathfrak{a}$ automorphisms.
Assuming that $m$ goes to infinity with $n$ in such a way that
\[
	\frac{1}{\mathfrak{a}}
    \frac{n^k}{(2m)^{\ell}}
    \frac{\chi^{2\ell}}{\Delta(\chi)^k}
    \prod_{v \in V(F)}
    \left(\frac{d}{d \chi}\right)^{\deg(v)} \Delta(\chi)
\]
has a positive limit, denoted by $\lambda$,
then the number of $F$-subgraphs in a random $(n,m,D)$-multigraph
follows a Poisson limit law of parameter $\lambda$.
\end{theorem}

\begin{proof}
The generating function of $F$-multigraphs is
\[
	F(z,w,(\delta_0, \delta_1, \ldots)) =
    \frac{1}{\mathfrak{a}} \bigg(\prod_{v \in V(F)} \delta_{\deg(v)} \bigg) w^{\ell} z^{k}.
\]
As in the proof of Theorem~\ref{th:subgraphs},
we replace, in Equation~\eqref{eq:distinguished_degree_constraints},
the generating function of the multigraph family
with the generating function of $F$-patchworks.
For the same reason as in Theorem~\ref{th:subgraphs_limit},
this generating function is then approximated by
$e^{(u-1) F(z,w,(\delta_0, \delta_1, \ldots))}$.
Thus, the asymptotic number of $(n,m,D)$-multigraphs with exactly $t$ $F$-subgraphs is
\[
    n! 2^m m! [z^n w^m u^t]
    \sum_{j \geq 0}
    (2j)! [x^{2j}]
    e^{(u-1) F(z,w,(\Delta(x), \Delta'(x), \Delta''(x), \ldots))}
    e^{z \Delta(x)}
    \frac{w^j}{2^j j!}.
\]
Its limit is extracted using the second part
of Theorem~\ref{th:distinguished_degree_constraints}.
\end{proof}

There are $2^m m!$ ways to orient and label the edges
of a graph with $m$ edges.
Thus, each graph matches $2^m m!$ multigraphs.
Conversely, consider a multigraph family $\mF$,
stable by multigraph automorphisms,
where each multigraph has $m$ edges,
and contains neither loops nor multiple edges.
Then $\mF$ can be partitioned into sets of sizes $2^m m!$,
each corresponding to a graph.
Thus, as proven by \cite{EdPR16},
counting graphs with degree constraints
can be achieved by removing loops and double edges
from multigraphs with degree constraints.
The following theorem describes the small subgraphs
of $(n,m,D)$-graphs, when $m = \bigO(n)$.
Its has been derived in the particular case
of regular graphs by \cite{Bo80} and \cite{Wo81},
and of graphs with degrees $1$ or $2$ by \cite{BP14}.
%

\begin{theorem} \label{th:graphs_degree_constraints}
Consider a random $(n,m,D)$-graph $G$
that satisfies the conditions stated at the beginning of the section.

\noindent
$\bullet$
Any connected graph
that is neither a tree nor a unicycle
is asymptotically almost surely not a subgraph of $G$.

\noindent $\bullet$
Denoting by $C_j$ a cycle of length $j$, and with $k$ a fixed integer $\geq 3$,
then $G[C_3], \ldots, G[C_k]$
are asymptotically independent Poisson random variables of means
\[
    \frac{1}{2j}
    \left( \frac{1}{2m/n}
    \frac{\chi^2 \Delta''(\chi)}{\Delta(\chi)}
    \right)^j
    \quad
    \text{ for each $3 \leq j \leq k$.}
\]
\end{theorem}

\bibliographystyle{abbrv}
\bibliography{biblio}

\newpage

  \appendix
  
  \begin{center}
  {\large \bf APPENDIX }
  \end{center}

\medskip
We present here the long version of the proofs
sketched in the previous sections, after an introduction to the technique
of translating combinatorial operations
into generating functions equations.

	\section{Analytic combinatorics}

    \paragraph{Symbolic method.}

The book of \cite{FS09}, available online,
provides an excellent introduction
to the techniques of analytic combinatorics.
The main idea is to associate to any combinatorial family $\mA$
of labeled objects a generating function
\[
    A(z) = \sum_{n \geq 0} a_n \frac{z^n}{n!},
\]
where $a_n$ denotes the number of objects of size $n$ in $\mA$.
In the present article,
to express the generating function
of interesting combinatorial families,
we apply the following dictionary to translate
combinatorial relations between the families
into analytic equations on their generating functions
\begin{itemize}
\item
\textbf{Disjoint union.}
If $\mA \cap \mB = \emptyset$, and $\mC = \mA \cup \mB$,
then $C(z) = A(z) + B(z)$.
\item
\textbf{Relabeled Cartesian product.}
In the relabeled Cartesian product $\mC = \mA \times \mB$,
we consider all relabellings of the pairs $(a,b)$,
with $a \in \mA$ and $b \in \mB$,
so that each label,
from $1$ to the sum of the sizes of $a$ and $b$,
appears exactly once.
We then have
\[
    C(z) = A(z) B(z).
\]
\item
\textbf{Set.}
A set of objects from $\mA$ has generating function $e^{A(z)}$.
\end{itemize}

	\paragraph{Laplace and saddle-point methods.}

We use in our proofs a simple case of the Laplace method
(see e.g.\ the book \emph{Analytic Combinatorics in Several Variables} of Pemantle and Wilson, 2013).

\begin{lemma}
Consider two entire functions $A(t)$ and $B(t)$,
where $B(t)$ is a positive function that reaches
its unique maximum at a point $r$, and $A(r) \neq 0$.
Then on any open interval $I$ (finite or infinite) that contains $r$, we have
\[
	\int_I A(t) B(t)^n dt
    \sim 
    A(r) \int_I B(t)^n dt
\]
whenever the integral is well defined.
\end{lemma}

The saddle-point method is a technique to compute
the asymptotics of the coefficients of a generating function.
The coefficient extraction is written as a Cauchy integral,
on which a Laplace method is applied.
There exist many variations of this technique.
We will use here the following lemma,
which is a particular case of Theorem~VIII.8 from \cite{FS09}.

\begin{lemma} \label{th:large_powers}
Consider two entire functions $A(z)$ and $B(z)$,
and a sequence of integers $N$ such that $N/n$
has a positive finite limit $\lambda$.
Assume there exists a positive solution $r$ to the equation
\[
	\frac{r B'(r)}{B(r)} = \lambda,
\]
such that $A(r) \neq 0$ and $r \left( \frac{r B'(r)}{B(r)} \right)' \neq 0$. Then
\[
	[z^N] A(z) B(z)^n \sim A(r) [z^N] B(z)^n.
\]
\end{lemma}

    \section{Proof of Theorem~\ref{th:distinguished}}

    \paragraph{Exact expression.}

Let us first prove the theorem for a family $\mH$
that is composed of the graphs isomorphic to some $(k,\ell)$-graph $H$,
that has a number $\mathfrak{a}$ of automorphisms.
Then the number of $H$-graphs is $k!/\mathfrak{a}$,
and the generating function of $\mH$ is
\[
    H(z,w) = \frac{1}{\mathfrak{a}} w^{\ell} z^k.
\]
A graph $G$ with $n$ vertices and one $H$-subgraph distinguished
can be decomposed as an $H$-graph and a set of isolated vertices,
plus some edges.
Applying the Symbolic Method,
the generating function of an $H$-graph and a set of isolated vertices is
\[
    H(z,w) e^z.
\]
If we assume that there are $n$ vertices,
we extract the coefficient in $z$ and obtain
\[
    n! [z^n] H(z,w) e^z.
\]
Then each pair of the $n$ vertices can be linked by an edge,
except the pairs already linked in the $H$-graph.
Thus, the number of edges that can be added is $\binom{n}{2} - \ell$.
For each of those, we decide either to add it,
or to not add it.
Thus, the generating function of graphs on $n$ vertices
with a distinguished $H$-graph, additional vertices,
and additional edges, is
\[
    n! [z^n] H(z,w) e^z (1+w)^{\binom{n}{2} - \ell}.
\]
Replacing $H(z,w)$ by its expression, we obtain
\[
    n! [z^n] H(z,w) e^z (1+w)^{\binom{n}{2} - \ell}
    =
    n! [z^n] \frac{1}{\mathfrak{a}} w^{\ell} z^k e^z (1+w)^{\binom{n}{2} - \ell}
    =
    n! [z^n] H \left(z, \frac{w}{1+w} \right) e^z (1+w)^{\binom{n}{2}}.
\]
Finally, we fixe the number of edges to $m$,
and obtain for the number of $(n,m)$-graphs
where one $H$-graph is distinguished the formula
\[
    n! [z^n w^m] H \left(z, \frac{w}{1+w} \right) e^z (1+w)^{\binom{n}{2}}.
\]
Now if $\mF$ is a general graph family,
its generating function is a sum,
for each graph $H$ that has at least one isomorphic copy in $\mF$,
of the generating function of the $H$-graphs
\[
    F(z,w) =
    \sum_{\text{there is an $H$-graph in $\mF$}}
    H(z,w).
\]
By linearity, we obtain for the number of $(n,m)$-graph
where one $\mF$-graph is distinguished the formula
\[
    \sum_{\text{there is a $H$-graph in $\mF$}}
    n! [z^n w^m] H \left(z, \frac{w}{1+w} \right) e^z (1+w)^{\binom{n}{2}}
    =
    n! [z^n w^m] F \left(z, \frac{w}{1+w} \right) e^z (1+w)^{\binom{n}{2}}.
\]

    \paragraph{Asymptotics.}

We now apply a bivariate saddle-point method
to extract the asymptotics (see \cite{FS09}).
In the previous expression, we apply the changes of variables
\[
	z \to n z, \qquad w \to \frac{m}{\binom{n}{2}} w,
\]
and obtain
\[
	\frac{n!}{n^n} \frac{\binom{n}{2}^m}{m^m}
    [z^n w^m]
    F \left(n z, \frac{m w / \binom{n}{2}}{1 + m w / \binom{n}{2}} \right)
    e^{n z} \left( 1 + \frac{m}{\binom{n}{2} w } \right)^{\binom{n}{2}}.
\]
Since the function
\[
    \frac{F \left(nz, m w / \binom{n}{2} \right)}{F\left(n, m / \binom{n}{2} \right)}
\]
converges uniformly to an analytic function $L(z,w)$,
we have $L(1,1) = 1$.
Furthermore, the function
\[
    \frac{F \left(nz, \frac{m w / \binom{n}{2}}{1 + m w / \binom{n}{2}} \right)}{F\left(n, m / \binom{n}{2} \right)}.
\]
converges uniformly to $L(z,w)$ as well,
because $m = \smallo \left(\binom{n}{2} \right)$.
Thus, there exists a sequence of analytic functions $(\epsilon_n(z,w))_{n \geq 0}$
converging uniformly to $0$ such that
\[
	F \left(nz, \frac{m w / \binom{n}{2}}{1 + m w / \binom{n}{2}} \right) =
    F\left(n, \frac{m}{\binom{n}{2}} \right)
    \left( L(z,w) + \epsilon_n(z,w) \right).
\]
The expression of the number of $(n,m)$-graphs
with one $\mF$-subgraph distinguished becomes
\[
	\frac{n!}{n^n} \frac{\binom{n}{2}^m}{m^m}
    F\left(n, \frac{m}{\binom{n}{2}} \right)
    [z^n w^m]
    \left( L(z,w) + \epsilon_n(z,w) \right)
    e^{n z} \left( 1 + \frac{m}{\binom{n}{2}} w \right)^{\binom{n}{2}}.
\]
By an $\exp-\log$ argument,
there exists a sequence of analytic functions $(\tilde{\epsilon}_n(w))_{n \geq 0})$
converging uniformly to $0$ such that
\[
	\left( 1 + \frac{m}{\binom{n}{2}} w \right)^{\binom{n}{2}} =
    e^{m w} ( 1 + \tilde{\epsilon}_n(w) ),
\]
so the expression becomes
\[
	\frac{n!}{n^n} \frac{\binom{n}{2}^m}{m^m}
    F\left(n, \frac{m}{\binom{n}{2}} \right)
    [z^n w^m]
    \left( L(z,w) + \epsilon_n(z,w) \right) (1 + \tilde{\epsilon}(w))
    e^{n z} e^{m w}.
\]
According to the saddle-point method,
the asymptotics is the same as the asymptotics of
\[
	\frac{n!}{n^n} \frac{\binom{n}{2}^m}{m^m}
    F\left(n, \frac{m}{\binom{n}{2}} \right)
    [z^n w^m]
    L(1,1)
    e^{n z} e^{m w}
    \sim
    \binom{\binom{n}{2}}{m}
    F\left(n, \frac{m}{\binom{n}{2}} \right).
\]

		\section{Proof of Corollary~\ref{th:distinguished_esp}}

Let $G$ denote a random $(n,m)$-graph.
If $H$ is a subgraph of $F$,
then $G$ contains $F$ only if it contains $H$, so
\[
	\esp(G[F]) \leq \esp(G[H]).
\]
Assume that $H$ is a densest subgraph of $F$,
with $k^{\star}$ vertices, $\ell^{\star}$ edges,
and a group of symmetries of size $\mathfrak{a}$.
Then there exist $k^{\star}!/\mathfrak{a}$ $H$-graphs,
so the generating function of the $H$-graphs is
\[
	H(z,w) = \frac{1}{\mathfrak{a}} w^{\ell^{\star}} z^{k^{\star}}.
\]
The expected number of $H$-subgraphs
in a random $(n,m)$-graph is equal
to the number of $(n,m)$-graphs
where one $H$-subgraph is distinguished,
divided by the total number $\binom{\binom{n}{2}}{m}$
of $(n,m)$-graphs.
Applying Theorem~\ref{th:distinguished},
we obtain, for $m = \bigO(n^{\alpha})$ with $0 < \alpha < 2$,
\[
	\esp(G[H]) \sim H \left(n, \frac{m}{\binom{n}{2}} \right) =
    \bigO \left(n^{\ell^{\star} ( \alpha-2+k^{\star} / \ell^{\star})} \right)
\]


    \section{Properties of $\chi$}

Consider the random variable $X$ that takes the value $d$
with probability proportional to $\frac{x^d}{d!}$,
for each $d$ from $D$.
Then
\[
    \proba(X=d) = \frac{1}{\Delta(x)} \frac{x^d}{d!}, \quad
    \esp(X) = \frac{x \Delta'(x)}{\Delta(x)}, \quad
    \esp(X(X-1)) = \left(\frac{x \Delta'(x)}{\Delta(x)} \right)'.
\]
Thus, the function $\frac{x \Delta'(x)}{\Delta(x)}$ is strictly increasing,
and it maps $[0,+\infty[$ to $[\min(D), \max(D)[$.
Conversely, when $\frac{2m}{n}$ has a limit in $]\min(D), \max(D)[$,
then $\chi$, defined implicitly by the relation
\[
    \frac{\chi \Delta'(\chi)}{\Delta(\chi)} = \frac{2m}{n},
\]
has a positive limit.

	\section{Proof of Theorem~\ref{th:distinguished_degree_constraints}}

	\paragraph{First part of the theorem.}

For completeness, we start by recalling a result from \cite{EdPR16}.
Each edge of a multigraph is oriented and labeled.
Thus, is can be represented as a triplet $(u,v,e)$,
where $u$ and $v$ denote the two linked vertices,
and $e$ the label of the edge.
The edge can then be cut into two labeled half-edges,
the first one hanging from $u$ and wearing the label $2e-1$,
the second one hanging from $v$ and wearing the label $2e$.
Cutting all the edges of a multigraph into half-edges,
we obtain a bijection between
the $(n,m,D)$-multigraphs,
and the sets of $n$ labeled vertices,
each coming with a set of size in $D$ of labeled half-edges,
such that the total number of half-edges is $2m$.
Thus the number of $(n,m,D)$-multigraphs is
\[
	\mg_{n,m,D} = (2m)! [x^{2m}] \Delta(x)^n.
\]

We now combine this half-edges construction
with the proof of Theorem~\ref{th:distinguished}.
Let $H$ denote a multigraph from $\mF$,
with $k$ vertices and $\ell$ edges,
and assume its automorphism group (both on vertices and edges) has size $\mathfrak{a}$.
Then the number of $H$-multigraphs is $2^{\ell} \ell! k! / \mathfrak{a}$,
and the generating function of the $H$-multigraphs is
\[
	H(z,w) = \frac{1}{\mathfrak{a}} w^k z^{\ell}.
\]
An $(n,m,D)$-multigraph where an $H$-subgraph is distinguished
can be uniquely decomposed as an $H$-multigraph,
a set of additional vertices,
and a set of labeled half-edges, each linked to a vertex,
and so that the number of half-edges and edges on each vertex
is an integer from the set $D$.
Hence, the number of half-edges attached to one of the additional vertices is in $D$,
while the number of half-edges attached to a vertex of degree $d$ from the $H$-multigraph
is in the set $D$, shifted by $-d$.
The generating function of multigraphs with degrees in $D$,
where one $H$-subgraph is distinguished, is then
\[
	\sum_{j \geq 0} (2j)!
    [x^{2j}] H(z,w,(\Delta(x), \Delta'(x), \ldots))
    e^{z \Delta(x)} \frac{w^j}{2^j j!}.
\]
Using the decomposition
\[
    F(z,w) =
    \sum_{\text{there is an $H$-graph in $\mF$}}
    H(z,w).
\]
and extracting the coefficient $n! 2^m m! [z^n w^m]$
concludes the proof of the first part of the theorem.

	\paragraph{Asymptotics.}

Using the classic formula
\[
	\frac{1}{\sqrt{2\pi}}
    \int_{-\infty}^{+\infty}
    t^{k} e^{-t^2/2} dt =
    \begin{cases}
    0 & \text{if $k$ is odd},\\
    \frac{(2m)!}{2^m m!} & \text{if $k = 2m$},
    \end{cases}
\]
we rewrite the sum of the previous expression as an integral
\begin{eqnarray*}
	&&\sum_{j \geq 0} (2j)!
    [x^{2j}] F(z,w,(\Delta(x), \Delta'(x), \ldots))
    e^{z \Delta(x)} \frac{w^j}{2^j j!}
    \\ \;\; &=&
    \sum_{j \geq 0}
    \frac{1}{\sqrt{2\pi}}
    \int_{-\infty}^{+\infty}
    w^j t^{2j} e^{-t^2/2} dt
    [x^{2j}] F(z,w,(\Delta(x), \Delta'(x), \ldots))
    e^{z \Delta(x)}
    \\ \;\; &=&
    \frac{1}{\sqrt{2\pi}}
    \int_{-\infty}^{+\infty}
    F(z,w,(\Delta(\sqrt{w} t), \Delta'(\sqrt{w} t), \ldots))
    e^{z \Delta(\sqrt{w} t)}
    e^{-t^2/2} dt,
\end{eqnarray*}
where switching the sum and the integral is licit
because we are working with entire analytic functions.
To obtain the number of $(n,m,D)$-multigraphs
where one $\mF$-subgraph is distinguished,
we extract the coefficient $[z^n w^m]$ from the previous expression
and multiply by $n! 2^m m!$
\[
	\frac{n! 2^m m!}{\sqrt{2\pi}} [z^n w^m]
    \int_{-\infty}^{+\infty}
    F(z,w,(\Delta(\sqrt{w} t), \Delta'(\sqrt{w} t), \ldots))
    e^{z \Delta(\sqrt{w} t)}
    e^{-t^2/2} dt.
\]
To simplify the saddle-point method, we apply successively the following changes of variables
\[
	z \to \frac{n z}{\Delta(\sqrt{w} t)},
    \qquad
    w = (x /t)^2,
    \qquad
    t \to \sqrt{2m} t.
\]
The expression becomes
\[
	\frac{n!}{n^n}
    2^m m! (2m)^{m+1/2}
    [z^n x^{2m}]
    \frac{1}{\sqrt{2\pi}}
    \int_{-\infty}^{+\infty}
    F\left( \frac{n z}{\Delta(x)} , \frac{x^2}{2m t^2},(\Delta(x), \Delta'(x), \ldots) \right)
    e^{n z}
    \Delta(x)^n
    t^{2m}
    e^{-m t^2} dt.
\]
The result follows from the saddle-point and Laplace methods.

    \section{Proof of Theorem~\ref{th:graphs_degree_constraints}}

    \paragraph{Second point.}

The generating function of cycles (graphs or multigraphs) of length $j$ is
\[
    C_j(z,w) = \frac{1}{2j} \delta_2^j w^j z^j.
\]
Observe that cycles of length $1$ are loops,
and cycles of length $2$ are double edges.
Corollary~\ref{th:distinguished_esp_degree_constraints}
proves that in asymptotically almost all $(n,m,D)$-multigraphs,
the cycles of length at most $k$ are disjoint.
Thus, the generating function of patchworks of cycles,
where the variable $u_j$ marks the cycles of length $j$,
can be approximated as
\[
    e^{\sum_{j=1}^k u_j C_j(z,w,(\delta_0, \delta_1, \ldots))} =
    e^{\sum_{j=1}^k \frac{u_j}{2j} (\delta_2 w z)^j)}.
\]
Applying Theorem~\ref{th:distinguished_degree_constraints},
the asymptotic number of $(n,m,D)$-multigraphs that contain exactly $t_j$
cycles of length $j$ for all $1 \leq j \leq k$ is
\[
    \mg_{n,m,D} \bigg[ \prod_{j=1}^k u_j^{t_j} \bigg]
    e^{\sum_{j=1}^k \frac{u_j-1}{2j}
    \left( \frac{\chi^2 \Delta''(\chi)}{\Delta(\chi)}
    \frac{n}{2m} \right)^j}.
\]
Choosing $t_1 = t_2 = 0$,
we obtain $(n,m,D)$-multigraphs without loops and double edges.
Each $(n,m,D)$-graph containing $t_j$ cycles of length $j$ for all $3 \leq j \leq k$,
matches exactly $2^m m!$ such multigraphs. So their asymptotic  number is
\[
    \mg_{n,m,D}
    e^{-\frac{1}{2}
    \frac{\chi^2 \Delta''(\chi)}{\Delta(\chi)}
    \frac{n}{2m}
    -\frac{1}{4}
    \left( \frac{\chi^2 \Delta''(\chi)}{\Delta(\chi)}
    \frac{n}{2m} \right)^2}
    \bigg[ \prod_{j=3}^k u_j^{t_j} \bigg]
    e^{\sum_{j=1}^k \frac{u_j-1}{2j}
    \left( \frac{\chi^2 \Delta''(\chi)}{\Delta(\chi)}
    \frac{n}{2m} \right)^j},
\]
where
\[
    \frac{\mg_{n,m,D}}{2^m m!}
    e^{-\frac{1}{2}
    \frac{\chi^2 \Delta''(\chi)}{\Delta(\chi)}
    \frac{n}{2m}
    -\frac{1}{4}
    \left( \frac{\chi^2 \Delta''(\chi)}{\Delta(\chi)}
    \frac{n}{2m} \right)^2}
\]
is equal to the number of $(n,m,D)$-graphs (are result already derived by \cite{EdPR16}).
Hence, the limit probability for a random $(n,m,D)$-graph to contain exactly
$t_j$ cycles of length $j$, for all $3 \leq j \leq k$, is
\[
    \prod_{j=3}^k \frac{\lambda_j^{t_j}}{t_j!} e^{-\lambda_j},
    \qquad \text{ where } \qquad
    \lambda_j = \frac{1}{2j}
    \left( \frac{\chi^2 \Delta''(\chi)}{\Delta(\chi)}
    \frac{n}{2m} \right)^j.
\]

    \paragraph{First point.}

If the connected multigraph $F$ is neither a tree nor a cycle,
then its essential density is greater than $1$.
Corollary~\ref{th:distinguished_esp_degree_constraints} implies that a random $(n,m,D)$-multigraph
asymptotically almost surely contains no copy of $F$.
We saw in the previous paragraph that the probability,
for a random $(n,m,D)$-multigraph,
to contain no loops or double edge
has a positive limit
\[
    e^{-\frac{1}{2}
    \frac{\chi^2 \Delta''(\chi)}{\Delta(\chi)}
    \frac{n}{2m}
    -\frac{1}{4}
    \left( \frac{\chi^2 \Delta''(\chi)}{\Delta(\chi)}
    \frac{n}{2m} \right)^2}.
\]
Since almost all $(n,m,D)$-multigraphs contain no $F$-sub(multi)graph,
almost all $(n,m,D)$-graphs contain no $F$-subgraph as well.

\end{document}